\documentclass[leqno,12pt]{article} 
\usepackage{amsmath, amssymb}
\usepackage{amsthm} 
\newcommand\supp{\mathop{\rm supp}}

\theoremstyle{plain} 
\newtheorem{theorem}{\indent\sc Theorem}[section]
\newtheorem{lemma}[theorem]{\indent\sc Lemma}

\newtheorem{proposition}[theorem]{\indent\sc Proposition}

\theoremstyle{definition} 
\newtheorem{definition}[theorem]{\indent\sc Definition}
\newtheorem{remark}[theorem]{\indent\sc Remark}

\title{A note about the discrete Riesz potential on $\mathbb{Z}^n$} 
\author{
\textsc{Pablo Rocha} 
}
\date{} 
\begin{document}

\maketitle


\begin{abstract}
In this note we prove that the discrete Riesz potential $I_{\alpha}$ defined on $\mathbb{Z}^n$ is a bounded operator 
$H^p(\mathbb{Z}^n) \to \ell^q(\mathbb{Z}^n)$ for $0 < p \leq 1$ and $\frac{1}{q} = \frac{1}{p} - \frac{\alpha}{n}$, where $0 < \alpha < n$.
\end{abstract}

\section{Introduction}

Given $0 < \alpha < n$, the discrete Riesz potential $I_{\alpha}$ on $\mathbb{Z}^n$ is formally defined by
\begin{equation} \label{Riesz potential}
(I_{\alpha}b)(j) = \sum_{i \in \mathbb{Z}^n \setminus \{ j \}} \frac{b(i)}{|i-j |^{n - \alpha}}, \,\,\,\,\,\, j \in \mathbb{Z}^n.
\end{equation}
The continuous counterpart of (\ref{Riesz potential}) is well known in the literature (see \cite{Weiss}, \cite{Stein}, \cite{Hedberg}, 
\cite{Taible}, \cite{Krantz}, \cite{Nakai}). In the discrete setting, Y. Kanjin and M. Satake in \cite[Theorem 4]{Kanjin} studied the operator given in (\ref{Riesz potential}) for the case $n=1$ and proved, for $0 < p \leq 1$ and $\frac{1}{q} = \frac{1}{p} - \frac{1}{\alpha}$, their $H^{p}(\mathbb{Z}) - H^{q}(\mathbb{Z})$ boundedness by means of the molecular decomposition of $H^p(\mathbb{Z})$. The $H^{p}(\mathbb{Z}) - \ell^{q}(\mathbb{Z})$ boundedness of (\ref{Riesz potential}), with $n=1$, was pointed out by the author in 
\cite[Remark 11]{Rocha}. Discrete operators analogous to (\ref{Riesz potential}) were studied by E. Stein and S. Wainger in \cite{Wainger} and
by D. Oberlin in \cite{Oberlin}.

The theory for Hardy spaces on $\mathbb{Z}^n$ was developed by S. Boza and M. Carro in \cite{Carro} (see also \cite{Boza}). There, the authors
gave a variety of distinct approaches to characterize the discrete Hardy spaces $H^p(\mathbb{Z}^n)$ analogous to the ones given for the Hardy spaces $H^p(\mathbb{R}^n)$. Ones of these characterizations is as follows: we consider the discrete Poisson kernel on $\mathbb{Z}^n$, which is defined by
\[
P_t^d(j) = C_n \frac{t}{(t^2 + |j|^2)^{(n+1)/2}}, \,\,\,\, t > 0, \,\, j \in \mathbb{Z}^n \setminus \{ {\bf 0} \}, \,\, P_t^d({\bf 0}) = 0,
\]
where $C_n$ is a normalized constant depending on the dimension. For $0 < p < \infty$ and a sequence $b = \{ b(i) \}_{i \in \mathbb{Z}^n}$ we say that $b$ belongs to $\ell^{p}(\mathbb{Z}^n)$ if
$$\| b \|_{\ell^p(\mathbb{Z}^n)} :=\left( \sum_{i \in \mathbb{Z}^n} |b(i)|^p \right)^{1/p} < \infty.$$ For $p=\infty$, we say that $b$ belongs to 
$\ell^{\infty}(\mathbb{Z}^n)$ if $$\|b \|_{\ell^\infty(\mathbb{Z}^n)} := \sup_{i \in \mathbb{Z}^n} |b(i)| < \infty.$$
Then, for $0 < p \leq 1$, we define
\[
H^p(\mathbb{Z}^n) = \left\{ b \in \ell^p(\mathbb{Z}^n) : \sup_{t>0} |(P_t^d \ast_{\mathbb{Z}^n} b)| \in \ell^p(\mathbb{Z}^n) \right\},
\]
with the "$p$-norm" given by
\[
\| b \|_{H^p(\mathbb{Z}^n)} := \| b \|_{\ell^p(\mathbb{Z}^n)} + \|\sup_{t>0} |(P_t^d \ast_{\mathbb{Z}^n} b)| \|_{\ell^p(\mathbb{Z}^n)}.
\]
In \cite{Carro}, S. Boza and M. Carro also gave an atomic characterization of $H^p(\mathbb{Z}^n)$ for $0 < p \leq 1$. Before establishing this result we recall the definition of $(p, \infty, d_p)$-atom in $H^p(\mathbb{Z}^n)$. 

\begin{definition} Let $0 < p \leq 1$ and $d_p := \lfloor n(p^{-1} - 1) \rfloor$. We say that a sequence 
$a = \{ a(j) \}_{j \in \mathbb{Z}^n}$ is an $(p, \infty, d_p)$-atom centered at a discrete cube $Q \subset \mathbb{Z}^n$ if the following three conditions hold:

(a1) $\supp a \subset Q$,

(a2) $\| a \|_{\ell^\infty(\mathbb{Z}^n)} \leq (\# Q)^{-1/p}$,

(a3) $\displaystyle{\sum_{j \in Q}} j^{\beta} a(j) = 0$ for every multi-index $\beta=(\beta_1, ..., \beta_n) \in \mathbb{N}_0^n$ with 
$\beta_1 + \cdot \cdot \cdot + \beta_n \leq d_p$.
\end{definition}

The atomic decomposition mentioned for $H^p(\mathbb{Z}^n)$ is established in the following theorem.

\begin{theorem} (\cite[Theorem 3.7]{Carro}) \label{atomic Hp} Let $0 < p \leq 1$, $d_p = \lfloor n (p^{-1} - 1) \rfloor$ and 
$b \in H^{p}(\mathbb{Z}^n)$. Then there exist a sequence of $(p, \infty, d_p)$-atoms $\{ a_k \}_{k=0}^{+\infty}$, a sequence of scalars 
$\{ \lambda_k \}_{k=0}^{+\infty}$ and a positive constant $C$, which depends only on $p$ and $n$, with 
$\sum_{k=0}^{+\infty} |\lambda_k |^{p} \leq C \| b \|_{H^{p}(\mathbb{Z}^n)}^{p}$ such that $b = \sum_{k=0}^{+\infty} \lambda_k a_k$, where the series converges in $H^{p}(\mathbb{Z}^n)$.
\end{theorem}

The main result of this note is contained in the following theorem, which will be proved in Section 3 via the atomic decomposition of 
$H^p(\mathbb{Z}^n)$ joint with some auxiliary results of Section 2.

{\sc Theorem} \ref{main result}.
{\it Let $0 < \alpha < n$ and let $I_{\alpha}$ be the discrete Riesz potential given by (\ref{Riesz potential}). Then, for
$0 < p \leq 1$ and $\frac{1}{q} = \frac{1}{p} - \frac{\alpha}{n}$
\[
\| I_{\alpha} \, b \|_{\ell^{q}(\mathbb{Z}^n)} \leq C \| b \|_{H^{p}(\mathbb{Z}^n)},
\]
where $C$ does not depend on $b$.}

\

{\bf Notation.} Throughout this paper, $C$ will denote a positive real constant not necessarily the same at each occurrence. We set 
$\mathbb{N}_0 = \mathbb{N} \cup \{0\}$. For every $A \subset \mathbb{Z}^n$, we denote by $\#A$ and $\chi_{A}$ the cardinality of the set $A$ and the characteristic sequence of $A$ on $\mathbb{Z}^n$ respectively. Given a real number $s \geq 0$, we write $\lfloor s \rfloor$ for the integer part of $s$. 

\section{Preliminaries}

We start recalling some basic facts about multiple series.  A multiple series is of the form
\begin{equation}
\displaystyle{\sum_{k \in \mathbb{Z}^{n}}} b(k), \label{multi-series}
\end{equation}
where $b(k) \in \mathbb{C}$ for each $k \in \mathbb{Z}^{n}$. There are many different ways to define the sum of a multiple series by means of partial sums.  In the literature the following two are the most common (see e.g. \cite{Alimov}, \cite{Weisz}): 

The \textit{Nth-quadratic partial sum} of the series in (\ref{multi-series}) is defined by
\[
S_N = \sum_{|k|_{\infty} \leq N} b(k),
\]
where $k =(k_1, ..., k_n) \in \mathbb{Z}^{n}$ and $|k|_{\infty} = \max \{ |k_i| : i=1, ..., n \}$. If $\displaystyle{\lim_{N \rightarrow \infty}} S_N$ exists we say that the series in (\ref{multi-series}) is \textit{quadratically convergent}.

\begin{remark} \label{bound series}
Given a nonnegative sequence $\{ b(k) \}_{k \in \mathbb{Z}^n}$, we have that
\[
\sum_{|k|_{\infty} \leq N} b(k) \leq \sum_{|k_n| \leq N} \cdot \cdot \cdot \sum_{|k_1| \leq N} b(k_1, ..., k_n), \,\,\,\,\,\, \forall \,
N \geq 1. 
\] 
\end{remark} 

The \textit{Nth-circular partial sum} of the series in (\ref{multi-series}) is defined by
\[
\widetilde{S}_N = \sum_{|k| \leq N} b(k),
\]
where $k =(k_1, ..., k_n) \in \mathbb{Z}^{n}$ and $|k| = (k_1^{2} + \cdot \cdot \cdot + k_n^{2})^{1/2}$. The series in (\ref{multi-series}) is \textit{circularly convergent} if $\displaystyle{\lim_{N \rightarrow \infty}} \widetilde{S}_N$ exists.

In general, the circular convergence and the quadratic convergence are not equivalent (see \cite{Alimov}, p. 7-8). However, if a series is absolutely convergent in the sense circular or quadratic, then both convergence are equivalent and their sums coincide. Since our results only involve absolutely convergent series we can use one or another definition as it suits.

The following result will be useful in the study of the discrete Riesz potential.

\begin{lemma}\label{series} If $\epsilon > 0$, then the multiple series
\begin{equation}
\sum_{k \in \mathbb{Z}^{n} \setminus \{ \bf{0} \}} \frac{1}{|k|^{n+\epsilon}} \label{serie 0}
\end{equation}
converges.
\end{lemma}

\begin{proof} Let $S_N$ be the $N$th-quadratic partial sum of the series in (\ref{serie 0}). Then, by Remark \ref{bound series}, we obtain
\[
S_{N} = \sum_{0 < | k |_{\infty} \leq N} \frac{1}{|k|^{n+\epsilon}} \leq 2^{n} \sum_{k_n =0}^{N} \cdot \cdot \cdot \sum_{k_2 =0}^{N} \sum_{k_1 =1}^{N} \frac{1}{(k_1^{2} + k_2^2 + \cdot \cdot \cdot + k_n^{2})^{\frac{n + \epsilon}{2}}}.
\]
On the other hand, it is clear that
\begin{equation} \label{estimate norm}
|k_1| + |k_2| + \cdot \cdot \cdot + |k_n| \leq n  (k_1^{2} + k_2^2 + \cdot \cdot \cdot + k_n^{2})^{1/2},
\end{equation}
by Multinomial Theorem, for every $k \in \mathbb{Z}^{n} \setminus \{ \bf{0} \}$, we have that
\begin{equation} \label{multinomial ineq}
(|k_1| + |k_2| + \cdot \cdot \cdot + |k_n|)^{n + \epsilon} \geq \max \{1,|k_1|^{1+ \frac{\epsilon}{n}}\} \cdot \max \{1,|k_2|^{1+ \frac{\epsilon}{n}}\} \cdot \cdot \cdot \max \{ 1, |k_n|^{1+ \frac{\epsilon}{n}} \}.
\end{equation}
So,
\[
\sum_{0 < | k |_{\infty} \leq N} \frac{1}{|k|^{n+\epsilon}} \leq 2^{n} \sum_{k_n =0}^{N} \cdot \cdot \cdot \sum_{k_2 =0}^{N} 
\sum_{k_1 =1}^{N} \frac{n^{n + \epsilon}}{ \max \{1,|k_1|^{1+ \frac{\epsilon}{n}}\} \cdot \max \{1,|k_2|^{1+ \frac{\epsilon}{n}}\} 
\cdot \cdot \cdot \max \{ 1, |k_n|^{1+ \frac{\epsilon}{n}} \}}
\]
\[
\leq 2^{n}n^{n + \epsilon} \left( 1 + \sum_{k_1=1}^{N} \frac{1}{k_1^{1 + \frac{\epsilon}{n}}} \right)^{n}, \,\,\, \forall \, N \geq 2.
\]
Finally, letting $N$ tend to infinity, we obtain
\[
\sum_{k \in \mathbb{Z}^{n} \setminus \{ \bf{0} \}} \frac{1}{|k|^{n+\epsilon}} := \lim_{N \rightarrow \infty} S_{N} \leq 2^{n}n^{n+\epsilon} \left( 1 + \sum_{k_1=1}^{\infty} \frac{1}{k_1^{1 + \frac{\epsilon}{n}}} \right)^{n} < \infty.
\]
\end{proof}  

A discrete cube $Q$ centered at $j =(j_1, ..., j_n)  \in \mathbb{Z}^n$ is of the form $Q = \prod_{1 \leq l \leq n} [j_l -m, j_l +m]$, where
for each $l=1, ..., n$, $[j_l -m, j_l +m] = \{ j_l - m, ..., j_l, ..., j_l + m \}$ with $m \in \mathbb{N}_0$. It is clear that 
$\# Q = (2m+1)^n$.

Let $0 \leq \alpha < n$, given a sequence $b = \{ b(i) \}_{i \in \mathbb{Z}^n}$ we define the centered fractional maximal sequence 
$M_{\alpha} b$ by
\[
(M_{\alpha} b)(j) = \sup_{Q \ni j} \frac{1}{\# Q^{1 - \frac{\alpha}{n}}} \sum_{i \in Q} |b(i)|, \,\,\,\,\,\, j \in \mathbb{Z}^n,
\]
where the supremum is taken over all discrete cubes $Q$ centered at $j$. We observe that if $\alpha = 0$, then $M_0 = M$ where $M$ is the centered discrete maximal operator.

The following result is a consequence of the harmonic analysis on spaces of homogeneous type applied to the space 
$(\mathbb{Z}^n, \mu, | \cdot |)$ where $\mu$ is the counting measure and $| \cdot |$ is the usual distance in $\mathbb{Z}^n$ 
(see \cite{St} or \cite{Deng}). We omit its proof.

\

\begin{theorem}\label{maximal} Let $b = \{ b(i) \}_{i \in \mathbb{Z}^n}$ be a sequence.
\begin{enumerate}
\item[(a)] If $b \in \ell^{1}(\mathbb{Z}^n)$, then for every $\alpha > 0$
\[
\#\{ j \in \mathbb{Z}^n : (Mb)(j) > \alpha \} \leq \frac{C}{\alpha} \| b \|_{\ell^{1}(\mathbb{Z}^n)},
\]
where $C$ is a positive constant which does not depend on $\alpha$ and $b$.

\item[(b)] If $b \in \ell^{p}(\mathbb{Z}^n)$, $1 < p \leq \infty$, then $Mb \in \ell^{p}(\mathbb{Z}^n)$ and
\[
\| Mb \|_{\ell^{p}(\mathbb{Z}^n)} \leq C \| b \|_{\ell^{p}(\mathbb{Z}^n)},
\]
where $C$ depends only on $p$ and $n$. 
\end{enumerate}
\end{theorem}

Next, we consider $0 < \alpha < n$, $1 < p < \frac{n}{\alpha}$ and $q$ defined by $\frac{1}{q} = \frac{1}{p} - \frac{\alpha}{n}$. Let $Q$ be a discrete cube centered at $j \in \mathbb{Z}^n$. By taking into account that $\frac{p}{q} + \frac{\alpha p}{n} = 1$, to apply the H\"older inequality with $\frac{n-\alpha}{n} + \frac{\alpha}{n} = 1$, we have
\begin{eqnarray*}
\frac{1}{\# Q^{1 - \frac{\alpha}{n}}} \sum_{i \in Q} |b(i)| &=& \frac{1}{\# Q^{1 - \frac{\alpha}{n}}} \sum_{i \in Q} |b(i)|^{\frac{p}{q}}
|b(i)|^{\frac{\alpha p}{n}} \\
&\leq& \left( \frac{1}{\# Q} \sum_{i \in Q} |b(i)|^{\frac{p}{q}(\frac{n}{n-\alpha})} \right)^{\frac{n-\alpha}{n}}
\left( \sum_{i \in \mathbb{Z}^n} |b(i)|^{p} \right)^{\frac{\alpha}{n}},
\end{eqnarray*}
to take the supremum over all cubes $Q$ centered at $j$ we obtain
\begin{equation} \label{fract max}
(M_{\alpha} b)(j) \leq \left[  M \left(|b|^{\frac{p}{q}(\frac{n}{n-\alpha})} \right)(j) \right]^{\frac{n-\alpha}{n}} 
\left( \sum_{i \in \mathbb{Z}^n} |b(i)|^p \right)^{\frac{\alpha}{n}}, \,\,\,\,\,\, j \in \mathbb{Z}^n. 
\end{equation}
Now, the pointwise estimate in (\ref{fract max}) and Theorem \ref{maximal} lead to the following result.

\begin{proposition} \label{fract max 2}
Let $0 < \alpha < n$. If $1 < p < \frac{n}{\alpha}$ and $\frac{1}{q} = \frac{1}{p} - \frac{\alpha}{n}$, then
\[
\| M_{\alpha} b \|_{\ell^q(\mathbb{Z}^n)} \leq C \| b \|_{\ell^p(\mathbb{Z}^n)}, \,\,\,\, \forall \,\, b \in \ell^p(\mathbb{Z}^n).
\]
\end{proposition}

We conclude these preliminaries with the following supporting result on $\mathbb{Z}$.

\begin{proposition} \label{op J}
Given $0 < \gamma <1$ and a sequence $b = \{ b(i) \}_{i \in \mathbb{Z}}$, let $J_{\gamma}$ be the operator defined by
\[
(J_{\gamma}b)(j) = \sum_{i \in \mathbb{Z}} \frac{b(i)}{\max \{ 1, |i-j|^{1 - \gamma} \}}, \,\,\,\,\, j \in \mathbb{Z}.
\]
Then, for $1 < p < \gamma^{-1}$ and $\frac{1}{q} = \frac{1}{p} - \gamma$
\[
\| J_{\gamma} b \|_{\ell^q(\mathbb{Z})} \leq C \| b \|_{\ell^p(\mathbb{Z})}, \,\,\,\, \forall \, b \in \ell^p(\mathbb{Z}).
\]
\end{proposition}

\begin{proof} It is easy to check that $(J_{\gamma}b)(j) = b(j) + (I_{\gamma}b)(j)$ for every $j \in \mathbb{Z}$, where $I_{\gamma}$ is the discrete Riesz potential on $\mathbb{Z}$. Then, the proposition follows from the fact that $\ell^p(\mathbb{Z}) \subset \ell^q(\mathbb{Z})$ embeds continuously and that $I_{\gamma}$ is a bounded operator $\ell^p(\mathbb{Z}) \to \ell^q(\mathbb{Z})$ for $1 < p < \gamma^{-1}$ and 
$\frac{1}{q} = \frac{1}{p} - \gamma$ (see \cite{Hardy}, p. 288).
\end{proof}

\section{The $H^p(\mathbb{Z}^n) - \ell^q(\mathbb{Z}^n)$ boundedness of $I_{\alpha}$}

In this section we establish the $H^p(\mathbb{Z}^n) - \ell^q(\mathbb{Z}^n)$ boundedness of the discrete Riesz potential $I_{\alpha}$ 
on $\mathbb{Z}^n$. For them, we first start studying the $\ell^p(\mathbb{Z}^n) - \ell^q(\mathbb{Z}^n)$ boundedness of $I_{\alpha}$.

\begin{theorem} \label{Riesz bound}
For $0 < \alpha < n$, let $I_{\alpha}$ be the discrete Riesz potential given by (\ref{Riesz potential}) . If $1 < p < \frac{n}{\alpha}$, 
$\frac{1}{q} = \frac{1}{p} - \frac{\alpha}{n}$ and $b \in \ell^p(\mathbb{Z}^n)$, then
\begin{equation} \label{pointwise estim}
| (I_{\alpha} b)(j) | < \infty, \,\,\,\, \forall \,\, j \in \mathbb{Z}^n,
\end{equation}
and
\begin{equation} \label{lplq estim for Riesz}
\| I_{\alpha} b \|_{\ell^q(\mathbb{Z}^n)} \leq C \| b \|_{\ell^p(\mathbb{Z}^n)}.
\end{equation}
\end{theorem}

\begin{proof} Let $b \in \ell^p(\mathbb{Z}^n)$ with $1 < p < \frac{n}{\alpha}$, then $p' > \frac{n}{n-\alpha}$ and so 
$(n-\alpha) p' - n >0$. To apply the H\"older inequality and Lemma \ref{series} with $\epsilon := (n-\alpha) p' - n$, we obtain
\[
| (I_{\alpha} b)(j) | \leq \| b\|_{\ell^p(\mathbb{Z}^n)} \|\{ |i|^{-(n-\alpha)} \}\|_{\ell^{p'}(\mathbb{Z}^n \setminus \{ {\bf 0}\})}
< \infty, \,\,\,\, \forall \,\, j \in \mathbb{Z}^n.
\]
Then, (\ref{pointwise estim}) follows.

From the inequalities (\ref{estimate norm}) and (\ref{multinomial ineq}), with $n - \alpha$ instead of $n + \epsilon$, and Remark \ref{bound series}, we have for $j=(j_1, ..., j_n) \in \mathbb{Z}^n$
\[
|(I_{\alpha}b)(j)| \leq \sum_{i_n \in \mathbb{Z}} \cdot \cdot \cdot \sum_{i_1 \in \mathbb{Z}} \frac{|b(i_1, ..., i_n)|}{\max \{ 1, |i_1-j_1|^{1 - \alpha/n}\} \cdot \cdot \cdot \max \{ 1, |i_n-j_n|^{1 - \alpha/n}\}}.
\]
Now, Remark \ref{bound series} leads to
\begin{equation} \label{estim Iriesz}
\left(\sum_{j \in \mathbb{Z}^n}|(I_{\alpha}b)(j)|^q \right)^{1/q} \leq
\end{equation}
\[
\left[ \sum_{j_n \in \mathbb{Z}} \cdot \cdot \cdot \sum_{j_1 \in \mathbb{Z}} \left(\sum_{i_n \in \mathbb{Z}} \cdot \cdot \cdot \sum_{i_1 \in \mathbb{Z}} \frac{|b(i_1, ..., i_n)|}{\max \{ 1, |i_1-j_1|^{1 - \alpha/n}\} \cdot \cdot \cdot \max \{ 1, |i_n-j_n|^{1 - \alpha/n}\}} \right)^q \right]^{1/q}.
\]
Finally, (\ref{lplq estim for Riesz}) follows from Proposition \ref{op J} with $\gamma = \frac{\alpha}{n}$, the Minkowski's inequality for integrals on the $\sigma$-finite product measure space $\mathbb{Z} \times \mathbb{Z}$ with the counting measure, and an iterative argument applied on the right-hand side of the inequality that appears in (\ref{estim Iriesz}).
\end{proof}

\begin{remark}
Another proof of (\ref{lplq estim for Riesz}) was given in \cite[Proposition (a)]{Wainger}.
\end{remark}

\begin{theorem} \label{main result}
Let $0 < \alpha < n$ and let $I_{\alpha}$ be the discrete Riesz potential given by (\ref{Riesz potential}). Then, for
$0 < p \leq 1$ and $\frac{1}{q} = \frac{1}{p} - \frac{\alpha}{n}$
\[
\| I_{\alpha} \, b \|_{\ell^{q}(\mathbb{Z}^n)} \leq C \| b \|_{H^{p}(\mathbb{Z}^n)},
\]
where $C$ does not depend on $b$.
\end{theorem}

\begin{proof}
We take $p_0$ such that $1 < p_0 < \frac{n}{\alpha}$. By Theorem \ref{atomic Hp}, given $b \in H^{p}(\mathbb{Z}^n)$ we can write 
$b = \sum_k \lambda_k a_k$ where the $a_k$'s are $(p, \infty, d_p)$ atoms, the scalars $\lambda_k$ satisfies $\sum_{k} |\lambda_k |^{p} \leq C \| b \|_{H^{p}(\mathbb{Z}^n)}^{p}$  and the series converges in $H^{p}(\mathbb{Z}^n)$ and so in $\ell^{p_0}(\mathbb{Z}^n)$ since
$H^{p}(\mathbb{Z}^n) \subset \ell^{p}(\mathbb{Z}^n) \subset \ell^{p_0}(\mathbb{Z}^n)$ embed continuously. 
For $\frac{1}{q_0} = \frac{1}{p_0} - \frac{\alpha}{n}$, by Theorem \ref{Riesz bound}, $I_{\alpha}$ is a bounded operator 
$\ell^{p_0}(\mathbb{Z}^n) \to \ell^{q_0}(\mathbb{Z}^n)$. Since $b = \sum_k \lambda_k a_k$ in $\ell^{p_0}(\mathbb{Z}^n)$, we have that 
$(I_{\alpha} \, b)(j) = \sum_{k} \lambda_k (I_{\alpha} \, a_k)(j)$ for all $j \in \mathbb{Z}^n$, and thus
\begin{equation}
|(I_{\alpha} \, b)(j)| \leq \sum_{k} |\lambda_k| |(I_{\alpha} \, a_k)(j)|, \,\,\,\, \forall \, j \in \mathbb{Z}^n. \label{puntual}
\end{equation}
Then for $1 \leq q$, by (\ref{puntual}) and Minkowski's integral inequality on $\sigma$-finite measure spaces, we have
\begin{equation} \label{mink ineq}
\|I_{\alpha} \, b \|_{\ell^{q}(\mathbb{Z}^n)} \leq \sum_{k} |\lambda_k| \|I_{\alpha} \, a_k \|_{\ell^{q}(\mathbb{Z}^n)};
\end{equation}
now for $0 < q < 1$, from (\ref{puntual}), it is easy to check that
\begin{equation} \label{q ineq}
\|I_{\alpha} \, b \|_{\ell^{q}(\mathbb{Z}^n)}^q \leq \sum_{k} |\lambda_k|^q \|I_{\alpha} \, a_k \|_{\ell^{q}(\mathbb{Z}^n)}^q.
\end{equation}
Thus, if for $0 < p \leq 1$ and $\frac{1}{q}= \frac{1}{p} - \frac{\alpha}{n}$ we see that $\|I_{\alpha} \, a_k \|_{\ell^{q}(\mathbb{Z}^n)} \leq C$, with $C$ independent of the $(p, \infty, d_p)$-atom $a_k$, then the estimate (\ref{mink ineq}) or (\ref{q ineq}) according to the case and the fact that 
$\sum_{k} |\lambda_k |^{p} \leq C \| b \|_{H^{p}(\mathbb{Z}^n)}^{p}$ lead to 
\[
\|I_{\alpha} \, b \|_{\ell^{q}(\mathbb{Z}^n)} \leq C \left( \sum_{k} |\lambda_k|^{\min\{1, q \}} \right)^{\frac{1}{\min\{1, q \}}} \leq C \left( \sum_{k} |\lambda_k |^{p} \right)^{1/p} \leq C\| b \|_{H^{p}(\mathbb{Z}^n)}.
\] 
Being $b$ an arbitrary element of $H^{p}(\mathbb{Z}^n)$, the theorem follows.

To conclude the proof we will prove that for $0 < p \leq 1$ and $\frac{1}{q}= \frac{1}{p} - \frac{\alpha}{n}$ there exists an universal constant $C > 0$, which depends on $\alpha$, $n$, $p$ and $q$ only, such that 
\begin{equation} \label{uniform estimate}
\|I_{\alpha} \, a \|_{\ell^{q}(\mathbb{Z}^n)} \leq C, \,\,\,\, \textit{for all} \,\, (p, \infty, d_p) - \textit{atom} \,\, a=\{ a(i) \}. 
\end{equation}
To prove (\ref{uniform estimate}), let $a(\cdot)$ be an atom centered at the cube 
$Q_{k^0}= \prod_{1 \leq l \leq n}[ k^0_l - m, k^0_l + m ]$. We put
$4\lfloor \sqrt{n} \rfloor Q_{k^0}=\prod_{1 \leq l \leq n} \left[ k^0_l - 4\lfloor \sqrt{n} \rfloor m, k^0_l + 4\lfloor \sqrt{n} \rfloor m \right]$. So
\begin{equation} \label{sum2}
\sum_{j \in \mathbb{Z}^n} |(I_{\alpha} \, a)(j)|^{q} = \sum_{j \in 4\lfloor \sqrt{n} \rfloor Q_{k^0}} |(I_{\alpha} \, a)(j)|^{q} + 
\sum_{j \in \mathbb{Z}^n \setminus 4\lfloor \sqrt{n} \rfloor Q_{k^0}} |(I_{\alpha} \, a)(j)|^{q}.
\end{equation} 
To estimate the first sum, we apply H\"older inequality with respect
the exponent $\frac{q_0}{q}$, then from Theorem \ref{Riesz bound}, the size condition (a2) on the atom $a(\cdot)$, and since 
$\frac{1}{p} - \frac{1}{q} = \frac{1}{p_0} - \frac{1}{q_0} = \frac{\alpha}{n}$ we have
\begin{equation} \label{estim C}
\sum_{j \in 4\lfloor \sqrt{n} \rfloor Q_{k^0}}|(I_{\alpha} \, a)(j)|^{q} \leq  
\left(\frac{8 \lfloor \sqrt{n} \rfloor + 1}{2} \right)^{n(q_0 - q)/q_0} \cdot
\left( \sum_{j \in \mathbb{Z}^n} |(I_{\alpha} \, a)(j)|^{q_0}\right)^{q/q_0} \cdot (\# Q_{k^0})^{(q_0-q)/q_0}
\end{equation}
\[
\leq C \left( \sum_{j \in Q_{k^0}}|a (j)|^{p_0}\right)^{q/p_0} \cdot (\# Q_{k^0})^{(q_0-q)/q_0}
\]
\[
\leq C \, (\# Q_{k^0})^{-q/p} \cdot (\# Q_{k^0})^{q/p_0} \cdot (\# Q_{k^0})^{(q_0-q)/q_0} = C,
\]
\\
with $C$ independent of $k^0$ and $m$.

To estimate the second sum in (\ref{sum2}), we put $N - 1 = \lfloor n(p^{-1} - 1) \rfloor$. In view of the moment condition (a3) of 
$a(\cdot)$ we have, for $j \in \mathbb{Z}^n \setminus 4 \lfloor \sqrt{n} \rfloor Q_{k^0}$, that
\[
(I_{\alpha} \, a)(j) = \sum_{i \in Q_{k^0}} |i-j|^{\alpha-n} \, a(i) = \sum_{i \in Q_{k^0}} [|i-j|^{\alpha-n} - q_{N}(i,j)] \, a(i),
\]
where $q_{N}(\, \cdot \,, j)$ is the degree $N - 1$ Taylor polynomial of the function $x \rightarrow |x-j|^{\alpha-n}$ expanded around $k^0$. 
By the standard estimate of the remainder term in the Taylor expansion there exists $\xi$ between $i$ and $k^0$ such that
\[
| |i-j|^{\alpha-n} - q_{N}(i, j) | \leq C |i - k^0 |^{N} | j - \xi|^{\alpha-n-N},
\]
for any $i \in Q_{k^0}$ and any $j \notin 4 \lfloor \sqrt{n} \rfloor Q_{k^0}$. Since $|j - \xi| \geq \displaystyle{\frac{|j - k^0|}{2}}$, we get
\[
| |i-j|^{\alpha-n} - q_{N}(i, j) | \leq C (2m+1)^{N} | j - k^0|^{\alpha-n-N}.
\]
This inequality and the condition (a2) of the atom $a(\cdot)$ allow us to conclude that
\[
|(I_{\alpha}a)(j)| \leq C \frac{(2m+1)^{n+N}}{(\# Q_{k^0})^{1/p}} | j - k^0|^{\alpha-n-N} \leq \frac{C}{(\# Q_{k^0})^{1/p}} 
\left[ M_{\frac{\alpha n}{n+N}} (\chi_{Q_{k^0}})(j) \right]^{\frac{n+N}{n}},
\]
for all $j \notin 4 \lfloor \sqrt{n} \rfloor Q_{k^0}$. Thus,
\begin{equation} \label{cota afuera}
\sum_{j \in \mathbb{Z}^n \setminus 4\lfloor \sqrt{n} \rfloor Q_{k^0}} |(I_{\alpha}a)(j)|^q \leq \frac{C}{(\# Q_{k^0})^{q/p}}
\sum_{j \in \mathbb{Z}^n} \left[ M_{\frac{\alpha n}{n+N}} (\chi_{Q_{k^0}})(j) \right]^{q \frac{n+N}{n}}.
\end{equation}
Since  $N-1= \lfloor n(\frac{1}{p}-1) \rfloor$, we have $q \frac{n+N}{n} > 1$. We write $\widetilde{q} = q \frac{n+N}{n}$ and let 
$\frac{1}{\widetilde{p}} = \frac{1}{\widetilde{q}} + \frac{\alpha}{n+N}$, so $\frac{\widetilde{p}}{\widetilde{q}} = \frac{p}{q}$. From
Proposition \ref{fract max 2}, we obtain
\begin{equation} \label{cota afuera 2}
\sum_{j \in \mathbb{Z}^n} \left[ M_{\frac{\alpha n}{n+N}} (\chi_{Q_{k^0}})(j) \right]^{q \frac{n+N}{n}} \leq
C \left( \sum_{j \in \mathbb{Z}^n} \chi_{Q_{k^0}}(j) \right)^{q/p} = C (\# Q_{k^0})^{q/p}.
\end{equation}
Now, (\ref{cota afuera}) and (\ref{cota afuera 2}) give
\begin{equation} \label{cota afuera 3}
\sum_{j \in \mathbb{Z}^n \setminus 4\lfloor \sqrt{n} \rfloor Q_{k^0}} |(I_{\alpha}a)(j)|^q \leq C.
\end{equation}
Finally, (\ref{estim C}) and (\ref{cota afuera 3}) lead to (\ref{uniform estimate}). Thus the proof is concluded.
\end{proof}

\begin{remark}
In \cite[Theorem 4]{Kanjin} for $n=1$, $0 < p \leq 1$, $1/q = 1/p - \alpha$ and $0 < \alpha < 1$, the $H^p(\mathbb{Z}) \to H^q(\mathbb{Z})$ boundedness of $I_{\alpha}$ was obtained. Recently, in \cite{Rocha2} we generalize this result on $\mathbb{Z}^n$. More precisely, we prove
the $H^p(\mathbb{Z}^n) \to H^q(\mathbb{Z}^n)$ boundedness of $I_{\alpha}$ for $n \geq 1$, $\frac{n-1}{n} < p \leq 1$, $1/q = 1/p - \alpha/n$ and $0 < \alpha < n$. To prove this result, as in \cite{Kanjin}, we furnish a molecular decomposition for the elements of 
$H^{p}(\mathbb{Z}^n)$ on the range $\frac{n-1}{n} < p \leq 1$. This decomposition joint with some results and ideas of the present work allow us to obtain such estimate.
\end{remark}

{\bf Acknowledgements.} I express my thanks to the referees for their useful suggestions and comments.

Pablo Rocha, Instituto de Matem\'atica (INMABB), Departamento de Matem\'atica, Universidad Nacional del Sur (UNS)-CONICET, Bah\'ia Blanca, Argentina. \\
{\it e-mail:} pablo.rocha@uns.edu.ar


\begin{thebibliography}{99}

\bibitem{Alimov} Sh. Alimov, R. R. Ashurov, A. K. Pulatov, {\it Multiple Fourier Series and Fourier Integrals}, Encyclopaedia of Math. Sci. 
Vol 42, Harmonic An. IV, Springer-Verlag, Berlin, (1992), 1-95.

\bibitem{Boza} S. Boza and M. Carro, {\it Discrete Hardy spaces}, Studia Math., 129 (1) (1998), 31-50.

\bibitem{Carro} S. Boza and M. Carro, {\it Hardy spaces on $\mathbb{Z}^N$}, Proc. R. Soc. Edinb., 132 A (1) (2002), 25-43.

\bibitem{Deng} D. Deng and Y. Han, Harmonic Analysis on Spaces of Homogeneous Type, Lecture Notes in Mathematics 1966, 
Springer-Verlag, Berlin Heidelberg, xii, 2009.

\bibitem{Hardy} G. H. Hardy, J. E. Littlewood and G. P\'olya, Inequalities, 2nd ed., Cambridge Univ. Press, London and new York, 1952.

\bibitem{Hedberg} L. Hedberg, {\it On certain convolution inequalities}, Proc. Am. Math. Soc. 36 (2) (1972), 505-510. 

\bibitem{Kanjin} Y. Kanjin and M. Satake, {\it Inequalities for discrete Hardy spaces}, Acta Math. Hungar., 89 (4) (2000), 301-313.

\bibitem{Krantz} S. Krantz, {\it Fractional integration on Hardy spaces}, Studia Mathematica, vol 73 (2) (1982), 87-94.

\bibitem{Nakai} E. Nakai, {\it Recent topics of fractional integrals}, Sugaku Expo. 20, No. 2 (2007), 215-235. 

\bibitem{Oberlin} D. Oberlin, {\it Two discrete fractional integrals}, Math. Res. Lett. 8, No. 1-2 (2001), 1-6. 

\bibitem{Rocha} P. Rocha, {\it Fractional series operators on discrete Hardy spaces}, Acta Math. Hungar., 168 (1) (2022), 202-216.

\bibitem{Rocha2} P. Rocha, {\it A molecular decomposition for $H^p(\mathbb{Z}^n)$}. To appear in J. Class. Anal.

\bibitem{Stein} E. Stein, Singular Integrals and Differentiability Properties of Functions, Princeton Univ. Press, 1970.

\bibitem{St} E. Stein, Harmonic Analysis: Real-Variable Methods, Orthogonality, and Oscillatory Integrals, Princeton Univ. Press,
1993.

\bibitem{Wainger} E. Stein and S. Wainger, {\it Discrete analogues in harmonic analysis. II: Fractional integration}, J. Anal. Math. 80 
(2000), 335-355.

\bibitem{Weiss} E. Stein and G. Weiss, {\it On the theory of harmonic functions of several variables}, Acta Math. 103 (1960), 25-62.

\bibitem{Taible} M. H. Taibleson and G. Weiss, {\it The molecular characterization of certain Hardy spaces}, Asterisque 77 (1980), 67-149.

\bibitem{Weisz} F. Weisz, {\it Summability of Multi-Dimensional Trigonometric Fourier Series}, Surveys in Approximation Theory (7) (2012),
1-179.

\end{thebibliography}
\end{document}